\documentclass[11pt]{amsart}

\usepackage{times}
\usepackage[english]{babel}
\usepackage[utf8]{inputenc}
\usepackage{amsmath,amssymb}
\usepackage{graphicx}
\usepackage[colorinlistoftodos]{todonotes}
\usepackage{color}

\def\VecSp{V^n}
\def\Real{\mathbb{R}}
\def\Comp{\mathbb{C}}
\def\alg{\mathfrak{A}}

\def\SO{\mathrm{SO}}
\def\euler{\mathcal{E}}
\def\ending{\hfill$\square$}

\newtheorem{thm}{Theorem}
\newtheorem{cor}[thm]{Corollary}

\newtheorem{lem}[thm]{Lemma}
\newtheorem{prop}[thm]{Proposition}

\newtheorem{exmp}[thm]{Example}
\title{
Cyclic Vectors of Associative Matrix Algebras and Reachability Criteria for Linear and Nonlinear Control Systems*
}

\author{Yuliy  Baryshnikov$^{1}$ and Andrey Sarychev$^{2}$}
\thanks{*This work was in part supported by AFOSR under Grant FA9550-11-1-0216}
\thanks{$^{1}$Yuliy Baryshnikov is with the Department of Electrical and Computer Engineering, University of Illinois at Urbana-Champaign, Urbana, IL, USA,
        {\tt\small ymb@illinois.edu}}%
\thanks{$^{2}$Andrey Sarychev is with the Department of Mathematics and Informatics U.Dini, University of Florence,
        Firenze, 50134, Italy
        {\tt\small asarychev@unifi.it}}%

\begin{document}

\begin{abstract}
Motivated by the controllability/reachability  problems for switched linear control systems and some classes of nonlinear (mechanical) control systems we address a related problem of existence of a cyclic vector for an associative (matrix) algebra. We provide a sufficient criterion for existence of cyclic vector and draw conclusions for controllability.
\end{abstract}

\maketitle
\thispagestyle{empty}
\pagestyle{empty}

\section{Introduction}

The classical criterion of R.~Kalman for controllability of a linear system
\[
\dot{x}=Ax+Bu, x\in \VecSp \cong\Real^n, u\in U\cong \Real^r,
\]
states that the system is controllable if and only if
\begin{equation}\label{Kc}
  \mathcal{R}=\mbox{span}\{A^\ell \mathcal{B}\}_{0\leq\ell<n}=\VecSp,
\end{equation}
where $\mathcal{B}=\mbox{span}\{Bu|\ u \in U\}$.

If $r=1$ (in the case of single input systems) controllability is equivalent to the fact that the linear span of the set $\{A^kb\}_{k=0,\ldots,n-1}$ is the whole space $\VecSp$, or, equivalently, to the {\it cyclicity} (see Subsection \ref{alg_form} for   definitions)  of the single controlled direction $b$ with respect to the unital associative (commutative) matrix algebra generated by $A$.

Recent developments (see for example \cite{Su01,Su02,Xi03,Sa09}) showed that  studying controllability and reachability in many control problems requires a generalization of this relation, leading to the question of cyclicity of certain vectors of {\em non commutative}  matrix algebras.

We are driven here primarily by the {\em design problem}: given the uncontrolled dynamics defined by a finite set of operators, one questions,  whether there exists a control operator $B$ such that its range $\mathcal{B}$, acted upon by the compositions of the  operators, spans $\VecSp$, or, equivalently, whether there exist  a cyclic subspace of the associative algebra, generated by these operators.

We start with illustrating this question by two types of examples. First, we address the {\em switched control systems}.
 Second, we discuss controllability properties of some  classes of nonlinear controlled systems,  such as  classical and multidimensional rigid body or, more generally, controlled mechanical
 systems (see  \cite{BL} and references therein).  Infinite-dimensional counterparts of these problems include studying  (approximate) controllability of Euler and Navier-Stokes  equations (\cite{AgS}).

 We apply to these problems  Lie algebraic or geometric control methods and manage to reformulate some of the methods in terms of cyclicity of subspaces for associative finite-dimensional algebras.

\subsection{Motivation I: reachability and controllability for switched linear control systems}

Recall that  switched linear control system has the form
 \begin{equation}\label{slcs}
   \dot{x}(t)=A_{s(t)}x(t)+B_{s(t)}u(t).
 \end{equation}

Here $x(t)$ is the trajectory in the  state space  $\VecSp \cong \mathbb{R}^n,\ n>1$, $u(t)$ is the control taking its values in $U\cong \mathbb{R}^r$ and $s(t)$ is the {\it switching law}.
The function $s(t)$ takes its values in a set $\{1, \ldots , m\}$ describing different {\it realizations} or {\em states} of the system.

It is customary (and in most cases it does not restrict generality) to assume that the discontinuities of $s$ form a discrete sequence of times and to talk about
{\it switching sequence} (\cite{Xi03}) $\sigma=\{(j_1,\tau_1), \ldots , (j_K,\tau_K)  \}$, where $i_k$ is the index of realization which acts on an interval of length $\tau_k$.

The system is called globally reachable
\footnote{An alternative terminology adapted in
nonlinear control is "globally controllable". What is called controllability in part of linear control bibliography, including \cite{Su01,Xi03} is also called null controllability}
 if for each given triple $t_0, \tilde x, \hat x$ there exists $T > t_0$, a
switching sequence $\sigma(t)$ and a (piecewise-continuous) control function $u(t)$ which steers the system from $x(t_0)=\tilde x$ to $x(T)=\hat x$.

A natural question is to provide necessary and sufficient conditions for reachability. Those are obtained  in
\cite{Su02,Xi03}.
 \begin{prop}
Reachable set of the system (\ref{slcs}) from the origin is a linear subspace $\mathcal{R}$ defined by
\begin{eqnarray}\label{res}
 \mathcal{R}=\mbox{span}\{A^{\ell_n}_{m_n}\cdots A^{\ell_1}_{m_1}\mathcal{B}_{m_1}\}_{0 \leq \ell_i < n, \ 1 \leq m_i \leq m, \ i=1, \ldots,n}, \nonumber
 \end{eqnarray}
 where $\mathcal{B}_{m_i}=\mbox{\rm Im} B_{m_i}$ is the subspace of $\VecSp$ spanned
 by the columns of the matrix $B_{m_i}$. \ending
 \end{prop}

\begin{cor} The system (\ref{slcs}) is globally reachable if
\begin{equation}\label{rv}
 \mathcal{R}=\VecSp. 
\end{equation}
\ending
\end{cor}

We will restrict ourselves to the case , where $B_1= \cdots =B_m=B$, so that
\begin{equation}\label{rcon}
\mathcal{R}:=
\mbox{span}\{A^{\ell_n}_{m_n}\cdots A^{\ell_1}_{m_1}\mathcal{B}\}_{0 \leq \ell_i < n, \ 1\leq m_i \leq m,\ i=1,\ldots,n}, \end{equation}
while the condition of reachability still takes form \eqref{rv}-\eqref{rcon}.

In a particular case of scalar control (single-input) $u(t)$ we set  $B=b \in \VecSp$
in the reachability condition \eqref{rv}-(\ref{rcon}).

\subsection{Algebraic reformulation}
\label{alg_form}

An equivalent reformulation of these conditions is the following.
Let $\VecSp,  \ n>1$ be  vector space, and  $\alg$ be a unital subalgebra of the associative  algebra\footnote{Recall that an algebra (over a field $k$, which in our case will be always either $\Real$ or $\Comp$), is a vector space over $k$ equipped with a multiplication operation. It is called unital is it contains a unity operator.}  $\mathcal{L}(\VecSp)$ of  linear operators on $\VecSp$, generated by  $A_1, \ldots , A_m \in \mathcal{L}(\VecSp)$.

A vector $v \in \VecSp$ is called {\em cyclic} whenever its orbit $\alg v=\{Av| \ A \in \alg\}$ coincides with $\VecSp$.
Similarly a vector subspace $\mathcal{B} \subset \VecSp$ is called cyclic, whenever  $\mbox{span}\{Av| \ A \in \alg, \ v \in \mathcal{B}\}=\VecSp$.

\begin{lem}
The reachability condition \eqref{rv}-(\ref{rcon}) is equivalent to the fact that $\mathcal{B}$ is a cyclic subspace of the algebra $\alg$.  $\square$
\end{lem}

The following {design} problem refers to the question of {\em existence } of a control of certain dimension,  rendering a switching system reachable.

{\bf Question.} {\it Given switched  linear uncontrolled dynamics - a set of $(n \times n)$-matrices $A_1, \ldots , A_m$, when is it possible to control the respective switched control systems  by a single (resp. $r$-dimensional) control?}

Equivalently, this question can be reformulated as follows:

{\it Does the associative algebra $\mathfrak{A}$  generated by the matrices $A_1, \ldots , A_m$ admit a cyclic vector (resp. $r$-dimen\-sional cyclic subspace)?}

According to Burnside theorem algebra $\mathfrak{A} \subset \mathcal{L}(\VecSp)$ is {\em transitive}, i.e. every nonzero vector is cyclic, if and only if $\alg=\mathcal{L}(\VecSp)$.

If a cyclic vector exists, it is hard to miss:

\begin{prop}
The set of cyclic vectors $b \in \VecSp$ for an associative matrix algebra $\alg$ is either empty or is an open dense set in $\VecSp$ (in fact, a complement to an algebraic hypersurface).\ending
\end{prop}

Similarly, if the set of cyclic $r$-dimensional subspaces of $\alg$ is non-empty,  then it  is an open dense subset of the corresponding Grassmanian. Therefore if switched linear dynamics is controllable by means of a $r$-dimensional  control
for some control operator $B$, then a {\em generic operator} would work.

\subsection{Motivation II: reachability/controllability for nonlinear mechanical control systems}
\label{mot2}

There are many examples of  (controlled) mechanical systems, which  are invariant with respect to a (linear) action of a (matrix) Lie group $G$. Often $G$ plays also the role of the configuration space (with $G$ acting on itself by left multiplication). A typical example is the rotation (attitude motion) of a rigid body (satellite) about its center of mass. Here the configuration space can be identified with $G=\SO(3)$ - the special orthogonal group of orientation preserving rotations in $\Real^3$. This construction generalizes to  {\em multidimensional rigid body (MRB)} with the configuration space $G=\SO(n)$\cite{Arn,FK}.

An infinite-dimensional analogue is Euler (or Navier-Stokes) model for the motion of incompressible fluid, whose configuration space is the (infinite-dimensional) group $G=\mathit{SDiff} M$ of volume-preserving diffeomorphisms of a domain $M$ (see \cite{Arn}).

A left-invariant Riemannian metric  on the group $G$ defines the (left-invariant) geodesic dynamics.
In this case the tangent bundle - the phase space - $TG$ can be identified with $G \times \mathfrak{g}$, where $\mathfrak{g}$ is the Lie algebra of $G$. The left-invariant Riemannian metric is defined by a scalar product $\langle \cdot , \cdot \rangle$ on the velocity space $\mathfrak{g} \cong T_eG$.

In the case of rotation of the rigid body this Riemannian metric is given by the quadratic form, whose coefficients are components of the
  {\it inertia tensor} of the body.

To describe the dynamics on the phase space we introduce
a bilinear map $\euler:\mathfrak{g} \times \mathfrak{g} \mapsto \mathfrak{g}$
 defined by  the formula (\cite{Arn}):
 \begin{equation}\label{definition_cal_B}
   \langle x, [y,z] \rangle =\langle \euler(x,y),z \rangle ,
 \end{equation}
where $[y,z]$ is the Lie bracket on $\mathfrak{g}$. In the finite-dimensional cases, we deal with, we may think of it as of the matrix commutator.

The evolution of mechanical system  is then given by a system of coupled kinematic and dynamic equations
\begin{equation}
  \dot{q}= q\circ v , \
  \dot{v} = \euler(v,v), \ q \in G, v \in \mathfrak{g},  \label{uncon}
\end{equation}
where $q\circ v$ stays for the left translation of $v$ by $q$.

We concentrate on  the second {\it dynamic} equation, which describes the evolution of the velocity (it is the {\em in-body} instantaneous angular velocity in the rigid body problem)
in the {\em Lie} algebra $\mathfrak{g}$, a linear space.


Besides the geodesic dynamics, the dissipation forces can be present. We model them  by a linear term $Dv$.

We introduce the control into the system via {\it generalized forces} $b_j$ so that it  becomes
\begin{equation}\label{dycon}
  \dot{v} = \mathcal{E}(v,v)+Dv+ \sum_{j=0}^r b_j u_j(t), \ (u_1, \ldots , u_r) \in \mathbb{R}^r.
\end{equation}
Often the forces $b_j$ depend on the configuration $q$, but for simplicity of the exposition
we choose them constant linearly independent vectors  $b_j \in \mathfrak{g},\ j=1, \ldots , r$.

Whenever $\mathcal{B}=\mbox{Span}\{b_j, \ j=1, \ldots ,r\}$ coincides with the whole  space $\mathfrak{g}$, the reachability of the equation  \eqref{dycon} is evident.
 If $r < \dim  \mathfrak{g}$, the reachability requires an examination.

An approach to studying reachability, advanced in \cite{AgS, BL, Sa09}, is
based on {\it Lie extensions}. For the equation (\ref{dycon}) the method of Lie extensions
allows one to add to a pair
of control vector fields $\hat b, b $  a new {\it extending control vector field}
\begin{equation}\label{liex}
b_e=\mathcal{E}(\hat b, b)+ \mathcal{E}(b,\hat b),
\end{equation}
  whenever
$\mathcal{E}(\hat b, \hat b)=0 \ (\mbox{mod} \ \mathcal{B})$.

The key point is that such  extension does not change the (closure of the) reachable set.
If after  series of  extensions  the  original and the extending control vector
fields would span $\mathfrak{g}$, then  reachability would be verified.

This provides  {\it Lie rank criterion of reachability}.
Note that the computation of the iterated extensions and tracing their spans is in general a difficult problem.

We  reformulate the iterated Lie extension procedure in terms of associative matrix algebras and cyclic subspaces.

Assume  that some of the control vector fields, say  $\hat b_1, \ldots , \hat  b_{r_0},$ satisfy the conditions
$\mathcal{E}(\hat b_i,\hat b_i)=0  \ (\mbox{mod} \ \mathcal{B})$,  while $b_{r_0+1}, \ldots , b_r$  are arbitrary.
Consider  the linear operators $A_j: \mathfrak{g} \mapsto \mathfrak{g}$:
\begin{equation}\label{AL}
A_j b = \mathcal{E}(\hat b_j,b)+\mathcal{E}(b,\hat b_j), \ j=1, \ldots , r_0.
\end{equation}

 By  definition application of the linear operator $A_j$ to each $b_k \in \mathcal{B}$ results in Lie extension \eqref{liex} and iterations of this application correspond to the iterations of the Lie extensions.
Therefore the Lie rank criterion for reachability can be formulated as the following
\begin{prop}
If the  $\mathcal{B} \subset \mathfrak{g}$ is a cyclic subspace
for the associative algebra generated by the
the linear maps \eqref{AL},  then (\ref{dycon}) is globally reachable. \ending
\end{prop}

 Once again we see that verification of controllability property in the nonlinear case has been transformed in a problem of a presence of a cyclic subspace  for an associative
 matrix (operator) algebra.

\subsection{Prior work}
Despite enormous amount  of literature on the controllability of linear (and nonlinear) systems, the applications of the theory of associative non-commutative algebras and their representations are rare. One can see certain precursors to the approach we undertook here in M.Arbib's papers (see, e.g. \cite{AZ}), but our focus on the general structure theory of the (non-commutative) associative algebras and their representations is missing there. The quiver-based approach, alluded to below also was glimpsed in the 70-80-ies (see \cite{Ha77, Hi87}), but again, mostly in the context of commutative algebras. The key triangular block decomposition results in \cite{RR} were not, it seems, used in the control-theoretic context before.

It should be noted, however, that the general theory of {\em Lie algebras} in the context of (often, non-linear) controllability has been used for a  long time (see, e.g. \cite{Br73}). The role of the solvability of the Lie algebras generated by the linear operators, which determine a switching dynamics, was emphasized in \cite{AgBL}.

The examples that motivated our study are coming from well-established areas. Thus, the theory of switched systems has developed into a comprehensive area with a variety of results known, see \cite{L}. The questions of controllability of switched linear systems were addressed recently in \cite{Su01,Su02,Xi03}; our focus on the design problems, and the deployment of the toolbox of the theory of associative algebras is new. Similarly, the questions of controllability in equivariant mechanical systems, including the infinite-dimensional ones, like Euler and Navier-Stokes systems of fluid dynamics) attracted a lot of attention recently (see \cite{AgS,BL,Sa09}), but without the algebraic formalism we exploit here.

\subsection{Plan of the paper}
The rest of the note is organized as follows: In Section \ref{cyclic} we introduce the necessary algebraic results (in particular, the key block-triangularization theorem and formulate our main results. In Section \ref{seccor} several implications of the main theorems are deduced, resulting in (reasonably) practical conditions for cyclicity of (representations) of algebras. Several examples of applications of such conditions for the (generalized) controlled rigid bodies are given in Section \ref{sec_ex}. Appendix with the proof of the main result concludes the note.

\section{Cyclic vectors for an associative finite-dimensional algebra}\label{cyclic}
We will work over the complex numbers (all results are valid over the reals, but the decompositions below require complexifications.)

Let $\VecSp \cong \mathbb{C}^n, \ n>1$ be  vector space, and  $\alg$ be a unital subalgebra of the algebra $\mathcal{L}(\VecSp)$ of the linear operators on $\VecSp$ and  $A_1, \ldots , A_s \in \mathcal{L}(\VecSp)$ be generators of $\alg$.

We consider $\VecSp$ as a {\em representation} $\rho$ of an associative algebra $\mathcal{A}$  with generators $a_1, \ldots , a_s$,
and $\rho(a_j)=A_j \in \mathcal{L}(\VecSp), j=1, \ldots ,r$.

Recall that representation $\VecSp$ is called {\em irreducible} if $\alg=\rho(\mathcal{A})$  does not possess
nontrivial invariant subspaces, which implies that every vector $x \in V$ is cyclic.
According  to Burnside's theorem \cite{RR} a nontrivial representation $\VecSp$ is irreducible
if and only if  $\rho(\mathfrak{A})=\mathcal{L}(\VecSp)$.

Representations that possess a cyclic vector are  called cyclic \cite{Et}. Ideally, we aim at understanding, what conditions on the generators $A_i$'s can guarantee the cyclicity. While in general this seems to be a hard problem, in several situations cyclicity can be derived from simple invariants of the representations.

\subsection{Triangular decompositions}
We start with an important result which introduces  block-triangular structure for  unital matrix algebra $\alg$.

\begin{prop}[\cite{RR}]
Let $\alg$ be a unital subalgebra of $\mathcal{L}(\VecSp)$. Then there exists a basis of $\VecSp$  and a partition $\VecSp=V_1 \oplus \cdots \oplus V_k$, such that in that basis matrices of operators $A \in \alg$ have a
block-triangular structure
\begin{equation}\label{btf}
\left(
  \begin{array}{ccccc}
    A_{11} & A_{12}  & \cdots & \cdots & A_{1k} \\
    0 & A_{22} & & \cdots & A_{2k} \\
    0 & 0 & A_{33} & \cdots & \cdots \\
    \cdots & \cdots & \cdots & \cdots & \cdots \\
    0 & 0 & 0 & 0 & A_{kk}\\
  \end{array}
\right)
\end{equation}
and  $\{1, \ldots ,k\}$  can be  partitioned
into a disjoint union $J_1 \bigcup \cdots \bigcup J_\ell$,
of {\em isomorphism classes}
so that:
\begin{enumerate}
  \item[i)] for each $i:  \  \{A_{ii}| \ A \in \alg\}=\mathcal{L}(V_i)$;
  \item[ii)] for each $i,j \in J_s$ and each  $A \in \alg$: $A_{ii}=A_{jj}$;
  \item[iii)] for each $i \in J_s$  there exists $A \in \alg$, such that $A_{ii}=E$, while $A_{jj}=0$ for $j \not\in J_s$;
  \item[iv)] if $i  \in s,\ j \in J_t$ and $J_s \neq J_t$, then $\{(A_{ii},A_{jj})| \ A \in \alg\}=\mathcal{L}(V_i) \times \mathcal{L}(V_j). \ $
 \end{enumerate}
\hfill$\square$
\end{prop}

\subsection{Main result}
Now we can formulate  sufficient condition for existence of  cyclic vectors in terms of block-diagonal structure \eqref{btf}.

\begin{thm}\label{k<d}
If for each isomorphism class $J_s$
\begin{equation}\label{dimgemult} \dim V_j \geq |J_s|,
\end{equation}
(here $j\in J_s$),
then the cyclic vectors of the algebra $\mathcal{A} \subset \mathcal{L}(V)$
form dense open subset in $V$.
\hfill$\square$
\end{thm}
Proof of the result  is presented in Appendix.

Note that the condition formulated in the theorem becomes also necessary whenever   block-triangular representation \eqref{btf} is block-diagonal.

Indeed, assume that  off-diagonal blocks vanish and the condition, formulated in Theorem \ref{k<d} is violated, say,  $\dim V_j  < |J_s|$ for $j \in J_s$.

Assume there is a cyclic vector $v \in \VecSp$. Splitting $v$ in a sum $v=v_1 + \cdots + v_k, v_i\in V_i,$ we note that each $v_j \neq 0$. Cyclicity means that for any $y=y_1+\cdots +y_k\in \VecSp$ there exists an element of $A \in \alg$ such that $Av=y$, that is
   $$ A_{ii}v_i=y_i, \ i=1, \ldots ,k, $$
and in particular
\begin{equation}\label{dis}
Fv_i=y_i, \forall i \in J_s,
\end{equation}
  where $F=A_{ii}$  for  $i \in J_s$ (as all representations in $J_s$ are isomorphic).

Since the representations of $\alg$ defined by the restrictions to $V_i, \ i \in J_s$ are all isomorphic, we can identify them,
thinking of all $v_i$ belonging to the same subspace $\tilde{V}$ and $F$ belonging to $\mathcal{L}(\tilde{V})$.
As $|J_s| > \dim \tilde{V}$, one sees that the vectors $Fv_i$ in \eqref{dis} are linearly dependent.

Let
 $\sum_{i \in J_s}\alpha_i Fv_i=0$. Then, choosing
$y_i$,  such that $\sum_{i \in J_s}\alpha_i y_i \neq 0$  we get an  incompatibility of the  system \eqref{dis}, and hence the vector $x$ fails to be  cyclic.

In the presence of the off-diagonal blocks the condition, formulated in the Theorem, ceases to be necessary as the following example shows.

\begin{exmp}
Consider two associative algebras $\alg, \tilde{\alg}$  of upper-triangular matrices
$$     \alg=\left\{\left(
      \begin{array}{ccc}
        a & b & c \\
        0 & a & b \\
        0 & 0 & a \\
      \end{array}
    \right)\right\}_{a,b,c}, \
\tilde{\alg}=\left\{\left(
      \begin{array}{ccc}
        a & b & c \\
        0 & a & 0 \\
        0 & 0 & a \\
      \end{array}
    \right)\right\}_{a,b,c},
 $$
whose block-diagonal structures coincide: there is a unique class $J_1$ with $|J_1|=3$. Since the
diagonal blocks $A_{ii}, \tilde{A}_{ii}$ are $1$-dimensional the assumption of the Proposition is violated for both algebras.

On the other side for  $\alg$ the vector $x=(0,0,1)^t$ is  cyclic, meanwhile for $\tilde{\alg}$ there are no cyclic vectors.
\hfill$\square$
\end{exmp}

One notes, that so far we dealt with the diagonal part of the block-triangular form \eqref{btf}. A far reaching generalization of this result, which fully involves the whole structure,  can be (hopefully) developed using the quiver constructions of representations of associative algebras \cite{ASS}. The applications of this set of ideas will be addressed elsewhere.

\section{Implications}
\label{seccor}
We have established a sufficient criterion for the existence of cyclic vectors for (a representation of) an associative algebra generated by  finite set of generators. When dealing with control-theoretic questions, it is desirable to formulate criteria in terms of the generators. This program is far from completeness; this Section contains some partial results.

The following result regards the case of single generator $A$ and is classical.
\begin{prop}\label{ex1}
Let unital algebra $\mathfrak{A} \subset \mathcal{L}(\VecSp)$ be generated by an operator $A$. The (commutative) algebra possesses
a cyclic vector if and only if  minimal  and   characteristic polynomials of $A$ coincide up to a
constant multiplier.  \hfill$\square$
\end{prop}

The necessity is obvious: whenever the degree of minimal polynomial is $k < n=\dim A$, then the dimension of $\alg$ is $k<n$ and for each $v \in \mathbb{R}^n$ the vectors $v,Av, \ldots , A^kv, \ldots $ span a proper subspace of  $\VecSp$. Sufficiency is a bit more delicate; to prove it one  has to involve invariant factorization or  Jordan canonical form.

 Following result  provides a necessary condition and a sufficient condition for the presence of a cyclic vector. We formulate it in {\it Popov-Hautus form}.

\begin{prop}\label{haut}
Assume that a unital associative algebra $\mathfrak{A}$ of operators  is generated by $A_1, \ldots , A_m \in  \mathcal{L}(\VecSp)$.
\begin{enumerate}
\item
The representation $\VecSp$ is cyclic only if for any collection $\mu_1 , \ldots , \mu_m \in \mathbb{C}$ the rank
 \begin{equation}\label{rkm1}
  \mbox{\rm rank}\left[A_1 -\mu_1 I|\cdots | A_m -\mu_m I     \right] \geq n-1.
   \end{equation}
\item If the corresponding {\em Lie algebra} generated by $A_1, \ldots A_m$ is solvable, then  condition
 (\ref{rkm1}) is also sufficient.  \hfill$\square$
\end{enumerate}
 \end{prop}

It does not take much longer to prove the following generalization to the multi-input case.

\begin{prop}\label{ex2} Assume that a unital associative algebra $\mathfrak{A}$ of operators  is generated by $A_1, \ldots , A_m \in  \mathcal{L}(\VecSp)$.

\begin{enumerate}
\item
The representation $\VecSp$ possesses a cyclic $r$-dimensional subspace $\mathcal{B}$,  only if for any collection $\mu_1 , \ldots , \mu_m \in \mathbb{C}$ the rank
\begin{equation}\label{rkr}
   \mbox{\rm rank}\left[A_1 -\mu_1 I|\cdots | A_m -\mu_m I     \right] \geq n-r.
 \end{equation}

\item If the corresponding Lie algebra generated by $A_1, \ldots A_m$ is solvable, then the condition
 (\ref{rkr}) is also sufficient. \hfill$\square$
\end{enumerate}
 \end{prop}

\begin{proof}

1) If \eqref{rkr} fails, i.e. for some $\tilde \mu =(\tilde \mu_1 , \ldots , \tilde \mu_m) \in \mathbb{C}^m$:  $$\mbox{\rm rank}\left[A_1 -\tilde \mu_1 I|\cdots | A_m -\tilde \mu_m I     \right] \leq n-r-1 ,$$  then there exist an $(r+1)$-dimensional subspace  $\mathcal{P} \subset V^{n*},$ such that
$$\forall p \in \mathcal{P}: \ p(A_j -\tilde \mu_j I)=0,  j=1, \ldots ,m. $$

For each $r$-dimensional subspace $\mathcal{B}$ its annihilator $\mathcal{B}^\perp \subset V^{n*} $ is $(n-r)$-dimensional and therefore has a nontrivial intersection with $\mathcal{P}$.

If $\tilde p \neq 0$ belongs to $\mathcal{P} \bigcap \mathcal{B}^\perp$, then
   $$\tilde p \cdot (A_{m_n}^{\ell_n}\cdots A_{m_1}^{\ell_1}\mathcal{B})=\tilde \mu_{m_n}^{\ell_n} \cdots \tilde \mu_{m_1}^{\ell_1} \tilde p \cdot \mathcal{B}=0 , $$
and hence $\tilde p$ annihilates $\mathcal{R}$ at the left-hand side of \eqref{rcon}, meaning that $\mathcal{B}$ fails to be cyclic.

2) Let both solvability assumption and \eqref{rkr} hold, while an $r$-dimensional subspace $\mathcal{B}$  fail to be cyclic, i.e. $\mathcal{R}$  in \eqref{rcon} is a proper subspace of $V^n$.

As far as $\mathcal{R} \subsetneq V^n$ is  invariant with respect to $A_1, \ldots , A_m$, then $\mathcal{R}^\perp \subset V^{n*}$ is  nontrivial invariant subspace for the adjoint
operators $A_1^{*}, \ldots,  A_m^{*}$.

Obviously $A_1^{*}, \ldots,  A_m^{*}$ generate solvable Lie algebra  and so do their restrictions  \linebreak $A_1^{*}|_{\mathcal{R}^\perp}, \ldots,  A_m^{*}|_{\mathcal{R}^\perp}$.

By Lie theorem \cite{GG} there is a common eigen(co)vector $p \in \mathcal{R}^\perp$ for the operators   
$A_1^{*}|_{\mathcal{R}^\perp}, \ldots,  A_m^{*}|_{\mathcal{R}^\perp}$,
i.e. for some $\tilde \mu=(\tilde \mu_1, \ldots , \tilde \mu _m) \in \mathbb{C}^m$:
\begin{equation}\label{joeg}
  p(A_j -\tilde \mu_j I)=0, \  j=1, \ldots ,m; \ p \mathcal{B}=0.
\end{equation}

The eigen(co)vectors,  satisfying (\ref{joeg}), form a subspace $\mathcal{P}_{\tilde \mu}$, which by assumption has  dimension  $\dim \mathcal{P}_{\tilde \mu} \leq r$.

The  set $M$ of the $m$-tuples   $\mu=(\mu_1 , \ldots ,  \mu_m) \in \mathbb{C}^m$, for which the rank in \eqref{rkr} is $<n$, is finite. Obviously $\tilde \mu \in M$.

By virtue of \eqref{joeg} $\mathcal{B}^\perp \bigcap \mathcal{P}_{\tilde \mu} \neq \{0\}$.
Denote by $\mathcal{P}^\perp_{\tilde \mu}\subset \VecSp$ the set of vectors annihilated by all $p \in \mathcal{P}_{\tilde \mu}$.
Then
$$\mathcal{P}^\perp_{\tilde \mu} + \mathcal{B} = (\mathcal{B}^\perp \bigcap \mathcal{P}_{\tilde \mu})^\perp \neq \VecSp,$$ i.e.  $\mathcal{B}$ is {\it not} transversal to $\mathcal{P}^\perp_{\tilde \mu}$.

Since $\dim \mathcal{P}^\perp_\mu \geq (n-r),\ \forall \mu \in M$, then by elementary transversality argument the set of $r$-dimensional subspaces $\mathcal{B}$, transversal to {\it all} $\mathcal{P}^\perp_\mu , \mu \in M$, is open dense in the respective Grassman manifold, hence generic $r$-dimensional subspace is cyclic.
\end{proof}

\begin{cor}
 Condition (\ref{rkr}) is necessary and sufficient for a presence of a cyclic subspace $\mathcal{B}$, whenever associative algebra $\mathfrak{A}$,   generated by $A_1, \ldots  , A_m$ is commutative.\ending
\end{cor}

\section{Examples}
\label{sec_ex}

We will apply the results of the previous section to the study of controllability of {\it multidimensional rigid body (MRB)},
which is a particular case of the model \eqref{dycon} described in  Subsection~\ref{mot2}.

The state space of the system is the Lie algebra $so(n)$ of skew-symmetric matrices.  The  dynamics of MRB is described by
{\it Euler-Frahm equation} (\cite{FK,Sa09})
$ \dot{M}=[C,\Omega^2]$. Here $\Omega$ is the angular velocity, $M$ is the momentum of MRB, defined by the relation $M=C\Omega+\Omega C=\mathcal{I}_C\Omega$, $C$ is a symmetric positive semidefinite matrix. The operator $\mathcal{I}_C$ is called  {\it inertia operator} of the MRB.

If one assumes $C$ positive definite, then  $\mathcal{I}_C$ is invertible, and we can write Euler-Frahm equation as
$\dot{\Omega}=\mathcal{I}^{-1}_C[C,\Omega^2]$.

One can choose a basis in such a way that $C$ becomes diagonal:\linebreak $C=\mbox{\rm diag}(C_1, \ldots, C_n)$.
Besides we will assume that MRB is dynamically asymmetric:  $0< C_1  < \cdots < C_n.$

We make use of  the equilibrium points of Euler-Frahm  equation. There are plenty of
 them; we pick "principal axes" - the  skew symmetric matrices $S^{ij}=\mathbf{1}_{ij}-\mathbf{1}_{ji}$
 with all but two elements $(ij)$ and $(ji)$ vanishing. Obviously matrices $(S^{ij})^2=-\mathbf{1}_{ii}-\mathbf{1}_{jj}$ are diagonal and commute  with the diagonal matrix $C$.

 Controlled Euler-Frahm equation with damping has form
\begin{equation}\label{cmrb}
 \dot{\Omega}=\mathcal{I}^{-1}_C[C,\Omega^2]+D\Omega +\sum_{k=1}^rb_j u_j(t), \
 b_j \in so(n), \ j=1, \ldots ,r.
\end{equation}


The symmetric bilinear form, corresponding to the quadratic term at the right-hand side of \eqref{cmrb} is
$$\euler(\Omega^1,\Omega^2)=
\frac{1}{2}\mathcal{I}^{-1}_C[C,\Omega_1\Omega_2+\Omega_2\Omega_1]. $$

  For our example we consider the controlled Euler-Frahm equation on  $6$-dimensional  algebra $so(4)$ of $4$-dimensional skew-symmetric matrices.
The number of controls $r=3$.

Let us choose the basis in $so(4)$ consisting of the $2$-element matrices - principal axes  $S^{12}, S^{13}, S^{14}$, $S^{23}, S^{24},S^{34}$.
The multiplication table for the bilinear form $\euler$ in this basis is defined by the relations
 \begin{eqnarray*}
  \euler(S^{ij}, S^{hk})=0,  \hbox{if} \ \{i,j\} \cap \{h,k\}= \emptyset;  \\
    \euler(S^{ij}, S^{ik})=c_{jk}S^{jk}, \  \hbox{where} \ c_{jk}=\frac{C_k-C_j}{C_j+C_k}.
 \end{eqnarray*}
and the convention $S^{ij}=-S^{ji}$ for all $i,j$.

According to the method described in Subsection \ref{mot2} (see also \cite{Sa09} for details) we choose  some of the controlled directions $b_j$ in \eqref{cmrb}  among $S^{ij}$.

\underline{Case 1.} Let  two controlled directions in  \eqref{cmrb} be $b_1=S^{12}, b_2=S^{23}$.
Taking the linearizations $\Lambda^{12},\Lambda^{23}$ of $\euler$ at  $S^{12},S^{23}$,
we wish to know whether there exist  $b \in so(4)$ - a cyclic vector for the algebra generated by $\Lambda^{12},\Lambda^{23}$.
According to \cite{Sa09} the MRB would be controllable in this case by the control "torque" $S^{12}u_1(t)+S^{23}u_2(t)+Lu(t)$.

Computing $\Lambda^{12},\Lambda^{23}$ in the chosen basis (using the multiplication table) we get:
$$\Lambda^{12}=\left(
                 \begin{array}{cccccc}
                   0& 0      & 0      & 0       & 0       & 0 \\
                   0 & 0     & 0      & -c_{13} & 0       & 0 \\
                   0 & 0     & 0      & 0       & -c_{14} & 0 \\
                   0 & c_{23}& 0      & 0       & 0       & 0 \\
                   0 & 0     & c_{24} & 0       & 0       & 0 \\
                   0 & 0     & 0      & 0       & 0       & 0 \\
                 \end{array}
              \right), $$

$$\Lambda^{23}=\left(
                 \begin{array}{cccccc}
                   0      & -c_{12}& 0   & 0       & 0       & 0 \\
                   c_{13} & 0      & 0   & 0       & 0       & 0 \\
                   0      & 0      & 0   & 0       & 0       & 0 \\
                   0      & 0      & 0   & 0       & 0       & 0 \\
                   0      & 0      & 0   & 0       & 0       & -c_{24} \\
                   0      & 0      & 0   & 0       & c_{34}  & 0 \\
                 \end{array}
               \right). $$

Evidently each of $\Lambda^{12}, \Lambda^{23}$   has double eigenvalue $0$,
and minimal polynomial of $5$th degree, what means that there are no cyclic vectors for the  associative algebras generated by $\Lambda^{12}$ {\it or}  by
$\Lambda^{23}$.

The operator $\Lambda^{12}+ \Lambda^{23}$ has double eigenvalue $0$.
We prove that its perturbation  $\Lambda^\varepsilon=\Lambda^{12}+ \Lambda^{23} +\varepsilon \Lambda^{12}\Lambda^{23}$ has simple eigenvalues for small $\varepsilon$
and for generic values of the principal moments of inertia $C_1,C_2,C_3,C_4$ of our MRB.

The characteristic polynomial of $\Lambda^\varepsilon$ has the form
$$P_\varepsilon(\zeta)=\zeta^6+p_4\zeta^4-\varepsilon p_3 \zeta^3+p_2\zeta^2-\varepsilon p_1 \zeta +\varepsilon^2 p_0.$$
Evidently it has double zero root for $\varepsilon=0$.

By elementary bifurcation theory  the first-order (in $\varepsilon$) perturbations $\zeta^\varepsilon_1,\zeta^\varepsilon_2$  of the  double zero root are to be  found from the bifurcation equation  (\cite{VT})
 $$ p_2\zeta^2-\varepsilon p_1 \zeta +\varepsilon^2 p_0=0.$$

The roots $\zeta^\varepsilon_1,\zeta^\varepsilon_2$ of this equation  are distinct  for small $\varepsilon$, if
the discriminant $\Delta=p_1^2-4p_2p_0 \neq 0$.

Direct computation shows that  $\Delta$ is a rational function of $C_1,C_,C_3,C_4$, representable as a
ratio $\Delta=\frac{\Delta_n(C_1,C_2,C_3,C_4)}{\Delta_d(C_1,C_2,C_3,C_4)}$ of two polynomials of 16th degree (with $\Delta_d>0$ for real $C$'s).
Therefore $\Delta \neq 0$ on the complement of an algebraic hypersurface
$\{\Delta_n(C_1,C_2,C_3,C_4)=0\}$, an open dense subset of the set of parameters $C_j$.

Thus for a generic choice of $C_j$ and for small $\varepsilon >0$ the matrix $\Lambda^\varepsilon$ has simple eigenvalues. It follows that for the associative algebra generated by $\Lambda^{12},\Lambda^{23}$ (or for its subalgebra generated by $\Lambda^\varepsilon$) there is an open dense set of  cyclic vectors $b \in so(4)$.

\underline{Case 2.} If we take $b_1=S^{12}, b_2=S^{34}$, then
the linearization $\Lambda^{34}$ of $\euler$ at  $S^{34}$
equals
$$\Lambda^{34}=\left(
                 \begin{array}{cccccc}
                   0& 0      & 0      & 0       & 0       & 0 \\
                   0 & 0     & -c_{13} &0       & 0       & 0 \\
                   0 & c_{14} & 0      & 0       & 0 & 0      \\
                   0 & 0     & 0      & 0       & -c_{23} & 0 \\
                   0 & 0     & 0      & c_{24}   & 0     & 0 \\
                   0 & 0     & 0      & 0        & 0       & 0 \\
                 \end{array}
               \right). $$
 Evidently $ p \Lambda^{12}= p \Lambda^{34}=0 $ for  all $p \in so^*(4)$, which in chosen basis satisfy the relations $\{p_i=0, \ 2 \leq i \leq 5\}$.

  Such $p$ form a $2$-dimensional subspace and hence  according to Proposition \ref{haut} there are no cyclic vectors for
 the associative algebra generated by $\Lambda^{12},\Lambda^{34}$.

\section{APPENDIX}

{\it Proof of  theorem~\ref{k<d}.} We proceed over the field $\mathbb{C}$ and use $V$ in place of $\VecSp$.

The formulation of  the theorem refers to the block-triangular form  \eqref{btf},  the proof refers more directly to
the structure of (the representation of) the associative algebra, which is manifested by \eqref{btf}.

{\bf 1.} First note  that one may consider semisimple  algebras. Indeed lacking semisimplicity for
$\mathfrak{A}$ means possessing a nontrivial radical $\mathfrak{N}=Rad(A)$, which in the case of finite-dimensional unital associative algebra is the largest nilpotent two-sided  ideal.

By Wedderburn theorem for associative algebras there exists a semisimple subalgebra $\mathfrak{S}$ of $\mathfrak{A}$, such that $\mathfrak{A}=\mathfrak{S}+\mathfrak{N}$.
The semisimple summand $\mathfrak{S}$ is not defined in unique way; different possible summands are related by similarities according to Malcev theorem (\cite{CR}).


For an algebra consisting  of matrices, having  block-triangular  form \eqref{btf} the following holds:
\begin{enumerate}
\item[i)] the radical $\mathfrak{N}$ consists of strictly block-triangular matrices, so that all blocks $N_{ii}$ vanish;
\item[ii)] the similarities preserve the block-diagonal parts of matrices.
\end{enumerate}

Obviously a cyclic vector for a semisimple summand $\mathfrak{S}$ is also a cyclic vector for the algebra $\mathfrak{A}$.

{\bf 2.} Any finite-dimensional semisimple algebra has finitely many (up to an isomorphism) irreducible representations $W_j$ and
\begin{equation}\label{ssd}
   \mathfrak{S}  \simeq \oplus^\ell_{j=1} \mathcal{L}(W_j).
\end{equation}
Any finite-dimensional representation of $\mathfrak{S} $ consists of a  finite number of copies of $W_j$.

 {\bf 3.} The units $\mathbf{e_j}, j=1, \ldots , \ell$  of the summands $\mathcal{L}(W_j)$ are idempotents of $\mathfrak{S}$,   which satisfy the  properties:

 i) $\mathbf{e_j}\mathbf{e_k}=\delta_{jk}\mathbf{e_k}$; \
 ii) $\mathbf{e_1}+ \cdots + \mathbf{e_\ell}=\mathbf{1_\mathfrak{S}}$ -  the unit of $\alg$.

 The vector spaces $e_jV$ are representations of the  summands $\mathcal{L}(W_j)$; conversely, to each $\ell$-ple of representations of $\mathcal{L}(W_j), \ j=1, \ldots , \ell,$ there corresponds a representation $V=e_1V \oplus \cdots \oplus e_\ell V$ of $\mathfrak{S}$.

 {\bf 4.} From now on we concentrate on a representation $e_jV$  of a single summand    $\mathcal{L}(W_j); \ \dim W_j=d_j$.
 The whole construction is just "glued together" from the constructions, accomplished for each summand.
  For the sake of brevity we put $V, W, d$ in place of $e_jV, W_j, d_j$.

 Irreducible representations of $\mathcal{L}(W)$ are isomorphic to\footnote{the left-side ideals of $\mathcal{L}(W)$, which are similar to the ideals of matrices with a single nonzero column, and hence to}  $\mathbb{C}^d$ . Any  representation  of $\mathcal{L}(W)$ is a finite direct sum of  $k$  copies of $\mathbb{C}^d$.
 The number $k$  is   the multiplicity of the respective diagonal block in \eqref{btf}: $k =|J_j|$.
Therefore $\dim V= kd$.

Recall that by assumption of the Theorem $k \leq d$.

 {\bf 5.} Let $E_{ij}, i,j=1, \ldots , d$ form a  basis in $\mathcal{L}(W)$, such that for some choice of a basis of $W$, $E_{ij}$ is $d \times d$-matrix  with unique nonzero (unit)  element at the intersection of $i$-th row and $j$-column.

 Evidently $E_{11}+ \cdots + E_{dd}$ is the identity $Id$, and
 $$\forall v \in V: \ v=\rho(E_{11})v+ \cdots +  \rho(E_{dd})v .$$
 Once again $E_{ii}$ are idempotents and $E_{ii}E_{jj}=\delta_{ij}Id$.
 Hence $\rho(E_{ii}):V \to V$ are projections and
 $$V=\rho(E_{11})V \oplus \cdots \oplus \rho(E_{dd})V. $$
 Besides $\rho(E_{ji}), \rho(E_{ij})$ are mutually inverse isomorphisms between the spaces \linebreak
 $ \rho(E_{ii})V$ and $ \rho(E_{jj})V$.
 Therefore $\dim \rho(E_{jj})V=k$ for each $j=1, \ldots , d$.

{\bf 6.}  Now we construct a  cyclic vector $x$   for the representation of $\mathcal{L}(W)$ on $V$.
We seek $x$ in the form $x=x_1 + \cdots + x_k$, where $x_j \in  \rho(E_{jj})V,\ j=1, \ldots k$.
   For the moment we {\it assume} that $x_1, \ldots , x_k$ are chosen in such a way, that  the vectors
  \begin{equation} \label{lix}
   \rho(E_{ij})x_j, \ i=1, \ldots , d;\  j=1, \ldots ,k \ \mbox{\rm are linearly independent,}
  \end{equation}
   and therefore span the $kd$-dimensional space $V$.

   To prove  that under assumption \eqref{lix} the vector $x$ is cyclic,  we act on $x$  by   matrix
 $\rho(B)=\sum_{i=1}^d \sum_{j=1}^d b_{ij}\rho(E_{ij}) $.
 Then
 $$\rho(B)x=\sum_{i=1}^d \sum_{j=1}^d \sum_{s=1}^k b_{ij}\rho(E_{ij})x_s,  $$
and as far as  $\rho(E_{ij})x_s=\delta_{js} \rho(E_{is})x_s$ we get
 $$\rho(B)x=\sum_{i=1}^d \sum_{s=1}^k b_{is}\rho(E_{is})x_s.  $$
 By virtue of assumption \eqref{lix} the  linear combinations at the right hand side span the whole $V$.

{\bf 7.} To find $x_1, \ldots x_k$, which satisfy \eqref{lix}, we proceed  by induction on $k$.

      For $k=1$ we pick any nonzero $x_1 \in \rho(E_{11})V $ and compute the vectors $\rho(E_{i1})x_1, \ i=1, \ldots , d,$ which are
   linearly independent as far $\rho(E_{i1})$ map isomorphically  $\rho(E_{11})V$  onto the direct summands
   $\rho(E_{ii})V$.

   Assume that one has   constructed  vectors $x_1,\ldots , x_{h-1} \ (h \leq k)$,  which satisfy \eqref{lix}.
   One can find  $x_{h} \in  \rho(E_{hh})V$, which is linearly independent with $h-1$ (linearly independent) vectors
   $\rho(E_{hj})x_j, \ j=1, \ldots , h-1$.
   Then   $\rho(E_{ih}), \ i=1, \ldots ,d$ map
   isomorphically $h$-dimen\-sional subspace $\mbox{Span}\{x_1, \ldots , x_h\}$, onto
   $h$-dimensional subspaces of the direct summands  $\rho(E_{ii})V$ and the step of induction is completed. \hfill$\square$

\nopagebreak

\section{ACKNOWLEDGMENT}

A.S. thanks the Coordinated Science Laboratory, University of Illinois at Urbana-Champaign for the hospitality during his visit in 2013.


\end{document}